\documentclass[a4paper,11pt]{article}

\usepackage{todonotes}
\usepackage[ruled,linesnumbered]{algorithm2e}
\usepackage{amsmath,amsthm,amssymb}
\usepackage{graphicx,subcaption,tikz,tkz-graph}
\usepackage{comment,xspace,paralist}							                              
\usepackage[shortlabels]{enumitem}	
\usepackage{hyperref}       
\hypersetup{
    colorlinks=true,
    citecolor = blue,
    linkcolor=black,
    filecolor=magenta,
    urlcolor=cyan,
    bookmarks=true,
}
\usepackage[margin=1in]{geometry}

\newtheorem{theorem}{Theorem}
\newtheorem{lemma}{Lemma}
\newtheorem{claim}{Claim}

\newcommand{\case}[1]{\noindent {\bf Case #1.}}	
\newcommand \e {\hfill {\tiny $\blacksquare$}}

\title{{\bf An Optimal $\chi$-Bound for ($P_6$, diamond)-Free Graphs}}

\author{Kathie Cameron \thanks{Department of Mathematics, Wilfrid Laurier University,
Waterloo, ON, Canada, N2L 3C5. Email: \texttt{kcameron@wlu.ca}. Research supported by the Natural Sciences and
Engineering Research Council of Canada (NSERC) grant RGPIN-2016-06517.}  
\and Shenwei Huang \thanks{Department of Mathematics, Wilfrid Laurier University,
Waterloo, ON, Canada, N2L 3C5. Email: \texttt{dynamichuang@gmail.com}. Research supported by the Natural Sciences and
Engineering Research Council of Canada (NSERC) grant RGPIN-2016-06517.} 
\and Owen Merkel \thanks{Department of Mathematics, Wilfrid Laurier University,
Waterloo, ON, Canada, N2L 3C5. Email: \texttt{owen.merkel@uwaterloo.ca}. Research supported by the Natural Sciences and
Engineering Research Council of Canada (NSERC) grant RGPIN-2016-06517 and an NSERC Undergraduate Student Research Award.}}

\date{August 29, 2018}

\begin{document}

\maketitle

\begin{abstract} 
Given two graphs $H_1$ and $H_2$, a graph $G$ is $(H_1,H_2)$-free if it contains no induced subgraph isomorphic to $H_1$ or $H_2$. 
Let $P_t$ be the path on $t$ vertices and $K_t$ be the complete graph on $t$ vertices.
The diamond is the graph obtained from $K_4$ by removing an edge.
In this paper we show that every ($P_6$, diamond)-free graph $G$
satisfies $\chi(G)\le \omega(G)+3$, where $\chi(G)$ and $\omega(G)$ are the chromatic number
and clique number of $G$, respectively. 
Our bound is attained by the complement of the famous 27-vertex Schl\"afli graph.
Our result unifies previously known results
on the existence of linear $\chi$-binding functions for several graph classes.
Our proof is based on a reduction via the Strong Perfect Graph Theorem to imperfect ($P_6$, diamond)-free graphs,
a careful analysis of the structure of those graphs, and a computer search that relies on a well-known  
characterization of 3-colourable $(P_6,K_3)$-free graphs.
\end{abstract}

\section{Introduction}

All graphs in this paper are finite and simple.
We say that a graph $G$ {\em contains} a graph $H$ if $H$ is
isomorphic to an induced subgraph of $G$.  A graph $G$ is
{\em $H$-free} if it does not contain $H$. 
For a family  $\mathcal{H}$ of graphs,
$G$ is {\em $\mathcal{H}$-free} if $G$ is $H$-free for every $H\in \mathcal{H}$.
When $\mathcal{H}$ consists of two graphs, we write
$(H_1,H_2)$-free instead of $\{H_1,H_2\}$-free.
As usual, $P_t$ and $C_s$ denote
the path on $t$ vertices and the cycle on $s$ vertices, respectively. The complete
graph on $n$ vertices is denoted by $K_n$.
The graph $K_3$ is also referred to as the {\em triangle}. 
Let the {\em diamond} be the graph obtained from $K_4$ by removing an edge.
For two graphs $G$ and $H$, we use $G+H$ to denote the \emph{disjoint union} of $G$ and $H$.
For a positive integer $r$, we use $rG$ to denote the disjoint union of $r$ copies of $G$.
The \emph{complement} of $G$ is denoted by $\overline{G}$.
A {\em clique} (resp. {\em stable set}) in a graph is a set of pairwise adjacent (resp. non-adjacent) vertices.
A \emph{$q$-colouring} of a graph $G$ is a function $\phi:V(G)\longrightarrow \{ 1, \ldots ,q\}$ such that
$\phi(u)\neq \phi(v)$ whenever $u$ and $v$ are adjacent in $G$.
Equivalently, a $q$-colouring of $G$ is a partition of $V(G)$ into $q$ stable sets.
A graph is {\em $q$-colourable} if it admits a $q$-colouring.
The \emph{chromatic number} of a graph $G$, denoted by
$\chi (G)$, is the minimum number $q$ for which $G$ is $q$-colourable.
The \emph{clique number} of $G$, denoted by $\omega(G)$, is the size of a largest clique in $G$.
Obviously, $\chi(G)\ge \omega(G)$ for any graph $G$.

A family $\mathcal{G}$ of graphs is said to be \emph{$\chi$-bounded} if 
there exists a function $f$ such that for every graph
$G\in \mathcal{G}$, every induced subgraph $H$ of $G$ satisfies
$\chi(H)\le f(\omega(H))$. The function $f$ is called a \emph{$\chi$-binding} function
for $\mathcal{G}$.
The notion of $\chi$-bounded families was introduced by Gy{\'a}rf{\'a}s \cite{Gy87} in 1987.
Since then it has received considerable attention for $\mathcal{H}$-free graphs.
 
We briefly review some results in this area.
A {\em hole} in a graph is an induced cycle of length at least 4.
An {\em antihole} is the complement of a hole.  A hole or antihole is {\em odd} or {\em even}
if it is of odd or even length, respectively.
The famous Strong Perfect Graph Theorem \cite{CRST06} says that the class of graphs without odd holes or odd antiholes
is $\chi$-bounded and the $\chi$-binding function is the identity function $f(x)=x$.
If we only forbid odd holes, then the resulting class remains $\chi$-bounded but the best known $\chi$-binding function
is double exponential \cite{SS16}. On the other hand, if even holes are forbidden, then a linear $\chi$-binding function exists \cite{BCHRS08}:
every even-hole-free graph $G$ satisfies $\chi(G)\le 2\omega(G)-1$.
In recent years, there has been an ongoing project led by Scott and Seymour that aims to determine the existence of $\chi$-binding functions
for classes of graphs without holes of various lengths. We refer the reader to the recent survey by Scott and Seymour \cite{SS18} 
for various nice results.
One thing to note is that most $\chi$-binding functions in this setting are exponential.

Another line of research is the study of $H$-free graphs for a fixed graph $H$. A classsical result of  Erd{\H{o}}s \cite{Er59}
shows that the class of $H$-free graphs is not $\chi$-bounded if $H$ contains a cycle. Gy{\'a}rf{\'a}s \cite{Gy73} conjectured
that the converse is also true (known as the Gy{\'a}rf{\'a}s  Conjecture), and proved the conjecture when $H=P_t$ \cite{Gy87}:
every $P_t$-free graph $G$ has $\chi(G)\le (t-1)^{\omega(G)-1}$.  Similar to results in \cite{SS18}, 
this $\chi$-binding function is exponential in $\omega(G)$. 
It is natural to ask the following question.

$\bullet$ Is it possible to improve the exponential bound for $P_t$-free graphs to a {\em polynomial} bound?

\noindent This turns out to be a very difficult question, and not much progress has been made over the past 30 years. 
It remains open whenever $t\ge 5$. (For $t\le 4$, $P_t$-free graphs are perfect and hence $f(x)=x$ is the $\chi$-binding function.)
Therefore, researchers have started to investigate subclasses of $P_t$-free graphs, hoping to discover techniques and methods that
would be useful for tackling the problem. A natural type of subclass is to forbid a second graph in addition to forbidding $P_t$.
For example, it was shown by Gaspers and Huang \cite{GH17} that every $(P_6,C_4)$-free graph $G$ has $\chi(G)\le \frac{3}{2}\omega(G)$.
This 3/2 bound was improved recently by Karthick and Maffray \cite{KMa18} to the optimal bound 5/4: 
every $(P_6,C_4)$-free graph $G$ has $\chi(G)\le \frac{5}{4} \omega(G)$.
In another work, Karthick and Maffray \cite{KM16} showed that 
every ($P_5$, diamond)-free graph $G$ satisfies $\chi(G)\le \omega(G)+1$.
Bharathi and Choudum \cite{BC16} gave a cubic $\chi$-binding function for the class of ($P_2+P_3$, diamond)-free graphs.
For the class of ($P_6$, diamond)-free graphs, a common superclass of ($P_5$, diamond)-free graphs 
and ($P_2+P_3$, diamond)-free graphs,
Karthick and Mishra \cite{KM18} proved that $f(x)=2x+5$ is a $\chi$-binding function, greatly improving
the result for ($P_2+P_3$, diamond)-free graphs. In the same paper, they also obtained an optimal $\chi$-bound for  ($P_6$, diamond)-free graphs 
when the clique number is 3: every ($P_6$, diamond, $K_4$)-free graph is 6-colourable.
For more results of this flavor, see \cite{CKS07,CKS08,HLR12,Wa80}.

\subsection*{Our Contributions} 
In this paper, we give an optimal $\chi$-bound for the class of 
($P_6$, diamond)-free graphs. We prove that each ($P_6$, diamond)-free graph $G$ satisfies $\chi(G)\le \omega(G)+3$
(\autoref{thm:main} in \autoref{sec:main}).
The bound is tight since it is attained by the complement of the famous 27-vertex Schl\"afli graph \cite{KM18}.
Our result unifies the results on the existence of linear $\chi$-binding functions for the class of ($P_5$, diamond)-free graphs \cite{KM16}, 
($P_2+P_3$, diamond)-free graphs \cite{BC16}, and ($P_6$, diamond)-free graphs \cite{KM18}, and answers an open question in \cite{KM18}.

The remainder of the paper is organized as follows. We present some preliminaries in \autoref{sec:pre}
and prove some structural properties of imperfect ($P_6$, diamond)-free graphs in \autoref{sec:imperfect}.
We prove our main result in \autoref{sec:main} and give some open problems in \autoref{sec:conclude}.

\section{Preliminaries}\label{sec:pre}

For general graph theory notation we follow \cite{BM08}.
Let $G=(V,E)$ be a graph. 
The \emph{neighbourhood} of a vertex $v$, denoted by $N_G(v)$, is the set of vertices adjacent to $v$.
For a set $X\subseteq V(G)$, let $N_G(X)=\bigcup_{v\in X}N_G(v)\setminus X$ and $N_G[X]=N_G(X)\cup X$.
The \emph{degree} of $v$, denoted by $d_G(v)$, is equal to $|N_G(v)|$.
For $x\in V$ and $S\subseteq V$, we denote by $N_S(x)$ the set of neighbours of $x$ that are in $S$,
i.e., $N_S(x)=N_G(x)\cap S$.
For $X,Y\subseteq V$, we say that $X$ is \emph{complete} (resp. \emph{anti-complete}) to $Y$
if every vertex in $X$ is adjacent (resp. non-adjacent) to every vertex in $Y$. For $x \in V$ and $Y \subseteq V$, we say $x$ is \emph{complete} (resp. \emph{anti-complete}) to $Y$ if $x$ is adjacent (resp. non-adjacent) to every vertex in $Y$.
A vertex subset $K\subseteq  V$ is a \emph{clique cutset} if $G-K$ has more components than $G$ and $K$
induces a clique. 
For $S\subseteq V$, the subgraph \emph{induced} by $S$ is denoted by $G[S]$.
We shall often write $S$ for $G[S]$ if the context is clear.
We say that a vertex $v$ {\em distinguishes} $u$ and $w$ if $v$ is adjacent to exactly one of $u$ and $w$.
A component of a graph is {\em trivial} if it has only one vertex, and {\em non-trivial} otherwise.

A graph $G$ is {\em perfect} if $\chi(H)=\omega(H)$ for each induced subgraph $H$ of $G$.
An {\em imperfect} graph is a graph that is not perfect.
One of the most celebrated theorems in graph theory is the Strong Perfect Graph Theorem \cite{CRST06}.

\begin{theorem}[\cite{CRST06}]\label{thm:SPGT}
A graph is perfect if and only if it does not contain an odd hole or an odd antihole as an induced subgraph.
\end{theorem}

Another useful result is a characterization of 3-colourable $(P_6,K_3)$-free graphs.
\begin{theorem}[\cite{RST02}]\label{thm:P6C3}
A $(P_6,K_3)$-free graph is $3$-colourable if and only if it does not contain the Gr{\"o}tzsch graph (see \autoref{fig:Grotzsch}) 
as an induced subgraph.
\end{theorem}

\begin{figure}[h!]
\centering
\begin{tikzpicture}[scale=0.6]
\tikzstyle{vertex}=[circle, draw, fill=black, inner sep=1pt, minimum size=5pt]
        \node[vertex, label=below:0](1) at (0,0) {};
        \node[vertex, label=1](2) at (0,2) {};
        \node[vertex, label=2](3) at ({2*cos(18)},{2*sin(18}) {};
        \node[vertex, label=3](4) at ({2*sin(36)},-{2*cos(36}) {};
        \node[vertex, label=4](5) at (-{2*sin(36)},-{2*cos(36}) {};
        \node[vertex, label=5](6) at (-{2*cos(18)},{2*sin(18}) {};
        \node[vertex, label=6](7) at (0,4) {};
        \node[vertex, label=7](8) at ({4*cos(18)},{4*sin(18}) {};
	 \node[vertex, label=below:8](9) at ({4*sin(36)},-{4*cos(36}) {};
	 \node[vertex, label=below:9](10) at (-{4*sin(36)},-{4*cos(36}) {};
	 \node[vertex, label=10](11) at (-{4*cos(18)},{4*sin(18}) {};
      
       \Edge(1)(2)
 	\Edge(1)(3)
 	\Edge(1)(4)
 	\Edge(1)(5)
	\Edge(1)(6)
	\Edge(2)(8)
	\Edge(2)(11)
	\Edge(3)(7)
	\Edge(3)(9)
	\Edge(4)(8)
	\Edge(4)(10)
	\Edge(5)(9)
	\Edge(5)(11)
	\Edge(6)(7)
	\Edge(6)(10)
	\Edge(7)(8)
	\Edge(8)(9)
	\Edge(9)(10)
	\Edge(10)(11)
	\Edge(7)(11)
\end{tikzpicture}
\caption{The Gr{\"o}tzsch graph.}\label{fig:Grotzsch}
\end{figure}
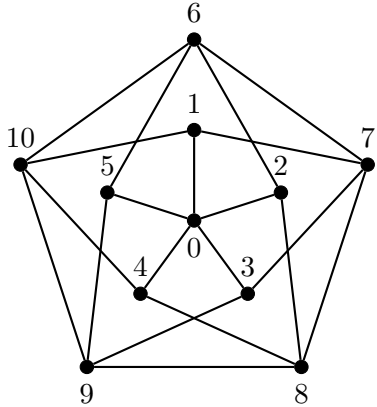

\section{Structure of Imperfect ($P_6$, diamond)-Free Graphs}\label{sec:imperfect}

In this section we study the structure of imperfect ($P_6$, diamond)-free graphs.
It follows from \autoref{thm:SPGT} that every imperfect ($P_6$, diamond)-free graph
contains an induced $C_5$. Let $G=(V,E)$ be an imperfect ($P_6$, diamond)-free graph and  let $Q=\{v_1, v_2, v_3, v_4, v_5\}$ induce a $C_5$ in $G$ with edges $v_iv_{i+1}$ for $i=1,\ldots,5$. Note that all indices are modulo 5. We partition $V(G) \setminus Q$ into the following subsets:
\begin{equation*} \label{eq1}
\begin{split}
A_i & = \{v\in V\setminus Q: N_{Q}(v)=\{v_i\}\}, \\
B_{i,i+1}  & = \{v\in V\setminus Q: N_{Q}(v)=\{v_i,v_{i+1}\}\}, \\
C_{i,i+2}  & = \{v\in V\setminus Q: N_{Q}(v)=\{v_i, v_{i+2}\}\}, \\
F_i  & = \{v\in V\setminus Q: N_{Q}(v)=\{v_i,v_{i-2},v_{i+2}\}\}, \\
Z & = \{v\in V\setminus Q: N_{Q}(v)=\emptyset\}. \\
\end{split}
\end{equation*}
Let $A=\bigcup_{i=1}^{5}A_i$, $B=\bigcup_{i=1}^{5}B_{i,i+1}$,  $C=\bigcup_{i=1}^{5}C_{i,i+2}$, 
and $F=\bigcup_{i=1}^{5}F_i$. Since $G$ is diamond-free, it follows that $N(Q)=A\cup B\cup C\cup F$ and thus 
$V(G)=Q\cup A\cup B\cup C\cup F\cup Z$.
We now prove a number of useful properties about those subsets.

\begin{enumerate}[label= (\arabic*)]

\item Each component of $A_i$ is a clique.\label{item:1}

This follows directly from the fact that $G$ is diamond-free. \e

\item The sets $A_i$ and $A_{i+1}$ are anti-complete.\label{item:2}

If $a_1\in A_i$ and $a_2\in A_{i+1}$ are adjacent, then $\{a_1,a_2,v_{i+1},v_{i+2},v_{i+3},v_{i+4}\}$ induces a $P_6$. \e

\item  The sets $A_i$ and $A_{i+2}$ are complete.\label{item:3}

If $a_1\in A_i$ and $a_2\in A_{i+2}$ are not adjacent, then $\{a_2,v_{i+2},v_{i+3},v_{i+4},v_i,a_1\}$ induces a $P_6$,
a contradiction. \e

\item Each $B_{i,i+1}$ is a clique.\label{item:4}

If $b_1,b_2\in B_{i,i+1}$ are not adjacent, then $\{b_1,b_2,v_i,v_{i+1}\}$ induces a diamond. \e

\item The set $B=B_{i,i+1}\cup B_{i+2,i+3}$ for some $i$.\label{item:5}

We show that for each $i$ either $B_{i,i+1}$ or $B_{i-1,i}$ is empty.
Suppose not. Let $b_1\in B_{i,i+1}$ and $b_2\in B_{i-1,i}$.
Then either $\{b_1,v_{i+1},v_{i+2},v_{i+3},v_{i+4},b_2\}$ induces a $P_6$ or
$\{b_1,b_2,v_i,v_{i+1}\}$ induces a diamond, depending on whether $b_1$ and $b_2$ are adjacent.
Therefore, the property holds. \e 

\item The set $B_{i,i+1}$ is anti-complete to $A_i\cup A_{i+1}$.\label{item:6}

By symmetry, it suffices to show that $B_{i,i+1}$ is anti-complete to $A_{i}$.
If $a\in A_i$ and $b\in B_{i,i+1}$ are adjacent, then $\{a,b,v_i,v_{i+1}\}$ induces a diamond. \e

\item The set $B_{i,i+1}$ is complete to $A_{i-1}\cup A_{i+2}$.\label{item:7}

By symmetry, it suffices to show that $B_{i,i+1}$ is complete to $A_{i+2}$.
If $a\in A_{i+2}$ and $b\in B_{i,i+1}$ are not adjacent, then $\{a,v_{i+2},v_{i+3},v_{i+4},v_i,b\}$ induces a $P_6$. \e

\item Each $C_{i,i+2}$ is a stable set. \label{item:8}

If $c_1,c_2\in C_{i,i+2}$ are adjacent, then $\{c_1,c_2,v_{i},v_{i+2}\}$ induces a diamond. \e

\item Each vertex in $C_{i,i+2}$ is either complete or anti-complete to each component of $A_i$ and $A_{i+2}$.\label{item:9}

If $c\in C_{i,i+2}$ is adjacent to $a_1\in A_i$ ($A_{i+2}$) but not adjacent to $a_2\in A_i$ ($A_{i+2}$)  with $a_1a_2\in E$, then 
$\{a_1,a_2,c,v_i\}$ ($\{a_1,a_2,c,v_{i+2}\}$) induces a diamond. \e

\item Each vertex in $C_{i,i+2}$ has at most one neighbour in each component of $A_{i+1}$, $A_{i+3}$ and $A_{i+4}$. \label{item:10}

Suppose that a vertex $c \in C_{i,i+2}$ has two neighbours $a_1$ and $a_2$ in the same component of $A_j$ where
$j \neq i$ and $j \neq i+1$. Since each component of $A_j$ is a clique by \ref{item:1}, $a_1a_2 \in E$. 
Then $\{c, a_1, a_2, v_j\}$ induces a diamond in $G$. \e

\item Each vertex in $C_{i,i+2}$ is anti-complete to each non-trivial component of $A_{i+1}$. \label{item:11}

Suppose that $c\in C_{i,i+2}$ has a neighbour $a_1$ in a non-trivial component of  $A_{i+1}$.
Let $a_2$ be a vertex in that component other than $a_1$. By \ref{item:10}, we have that $ca_2 \notin E$.
Then $\{a_2,a_1,c,v_{i+2},v_{i+3},v_{i+4}\}$ induces a $P_6$. \e

\item The set $C_{i,i+2}$ is anti-complete to $B_{j,j+1}$ if $j \neq i+3$.
Moreover each vertex in $C_{i,i+2}$ has at most one neighbour in $B_{i+3,i+4}$.\label{item:12}

Suppose that $c \in C_{i,i+2}$ is adjacent to some $b \in B_{j,j+1}$ for some $j\neq i+3$. 
Since $b$ and $c$ have exactly one common neighbour in $Q$, it follows that $\{b, c, v_{j}, v_{j+1}\}$ induces a diamond.
This proves the first part of the claim.
Suppose that $c$ is adjacent to two vertices $b_1, b_2 \in B_{i+3, i+4}$.
By \ref{item:4}, $b_1b_2\in E$. Then $\{c, b_1, b_2, v_{i+3}\}$ induces a diamond. \e

\item Each $F_i$ has at most one vertex. Moreover, $F$ is a  stable set.\label{item:13}

If $F_i$ contains two vertices $f_1$ and $f_2$, then either $\{v_i,v_{i+2},f_1,f_2\}$ or $\{v_{i-2},v_{i+2},f_1,f_2\}$ 
induces a diamond, depending on whether $f_1f_2\in E$. This proves the first part of the claim.
Let $f_1\in F_i$ and $f_2\in F_j$ with $i\neq j$. Note that there exists an index $k$ such that $v_k$ is a common neighbour
of $f_1$ and $f_2$ and $v_{k+1}$ is adjacent to exactly one of $f_1$ and $f_2$. 
If $f_1f_2\in E$, then $\{f_1,f_2,v_k,v_{k+1}\}$ induces a diamond. \e

\item The set $F_i$ is anti-complete to $A_{i+2}\cup A_{i+3}$.\label{item:14}

By symmetry, it suffices to show that $F_i$ is anti-complete to $A_{i+2}$.
If $f\in F_i$ is adjacent to $a\in A_{i+2}$, then $\{a,f,v_{i+2},v_{i+3}\}$ induces a diamond. \e 

\item Each vertex in $F_i$ is either complete or anti-complete to each component of $A_i$.\label{item:15}

If $f\in F_i$ is adjacent to $a_1\in A_i$ but not to $a_2\in A_i$ with $a_1a_2\in E$, then $\{a_1,a_2,f,v_i\}$ induces a diamond. \e

\item Each vertex in $F_i$ has at most one neighbour in each component of $A_{i+1}$ and $A_{i+4}$.\label{item:16}

If $f\in F_i$ has two neighbours $a_1$ and $a_2$ in the same component of $A_{i+1}$ ($A_{i+4}$),
then $\{f,a_1,a_2,v_{i+1}\}$ ($\{f,a_1,a_2,v_{i+4}\}$) induces a diamond. \e

\item The set $F_i$ is anti-complete to $B_{j,j+1}$ if $j \neq i+2$ and complete to $B_{j,j+1}$ if $j = i + 2$.\label{item:17}

If $f\in F_i$ is not adjacent to $b\in B_{i+2,i+3}$, then $\{f,b,v_{i+2},v_{i+3}\}$ induces a diamond. This proves the
second part of the claim. Note that $f\in F_i$ and $b\in B\setminus B_{i+2,i+3}$ have exactly one common neighbour in $Q$, say $v_k$.
If $fb\in E$, then $\{f,b,v_k,v_{k+1}\}$ or $\{f,b,v_k,v_{k-1}\}$ induces a diamond. \e

\item The set $F_i$ is anti-complete to $C_{j,j+2}$  if $j \neq i-1$. \label{item:18}

Let $f\in F_i$. Note that if $j\neq i-1$, then each vertex $c\in C_{j,j+2}$ is adjacent to exactly one of $v_{i+2}$ and $v_{i+3}$.
If $fc\in E$, then $\{v_{i+2}, v_{i+3}, f, c\}$ induces a diamond. \e

\item If $A_i$ is not stable, then $A_{i+2}=A_{i+3}=B_{i+1,i+2}=B_{i-1,i-2}=\emptyset$.\label{item:19}

Suppose that $A_i$ contains an edge $a_1a_2$. If there is a vertex $x$ in $A_{i+2}\cup A_{i+3}\cup B_{i+1,i+2}\cup B_{i-1,i-2}$,
then $x$ is adjacent to $a_1$ and $a_2$ by \ref{item:3} and \ref{item:7}. Then $\{x,a_1,a_2,v_i\}$ induces a diamond. \e

\item If $A_i$ is not empty, then each of $B_{i+1,i+2}$ and $B_{i-1,i-2}$ contains at most one vertex.\label{item:20}

Let $a\in A_i$. If $B_{i+1,i+2}$ (resp. $B_{i-1,i-2}$) contains two vertices $b_1$ and $b_2$, 
then $\{a,b_1,b_2,v_{i+1}\}$ (resp. $\{a,b_1,b_2,v_{i-1}\}$) induces a diamond by \ref{item:4} and \ref{item:7}. \e

\item The set $Z$ is anti-complete to $A\cup B$.\label{item:21}

Suppose that $z\in Z$ has a neighbour $x\in A\cup B$. By symmetry, we may assume that $x$ is adjacent to $v_i$ but
adjacent to none of $v_{i+2}$, $v_{i+3}$ and $v_{i+4}$. Then $\{z,x,v_i,v_{i+4},v_{i+3},v_{i+2}\}$ induces a $P_6$. \e
\end{enumerate}

\section{The Optimal $\chi$-Bound}\label{sec:main}

In this section, we derive an optimal $\chi$-bound for ($P_6$, diamond)-free graphs. An {\em atom} is a graph without clique cutsets. 
Two non-adjacent vertices $u$ and $v$ in a graph $G$ are {\em comparable} if $N(u)\subseteq N(v)$ or $N(v)\subseteq N(u)$. 
The following is the main result of this paper.

\begin{theorem}\label{thm:main}
Let $G$ be a ($P_6$, diamond)-free graph. Then $\chi(G)\le \omega(G)+3$.
\end{theorem}

\begin{proof}
We prove this by induction on $|V(G)|$.
If $G$ is disconnected, then we are done by applying the inductive hypothesis to each connected component of $G$.
If $G$ contains a clique cutset $S$ such that $G-S$ is the disjoint union of two subgraphs $H_1$ and $H_2$, then
it follows from the inductive hypothesis that $\chi(G)=\max\{\chi(G[V(H_1)\cup S]),\chi(G[V(H_2)\cup S])\}\le \omega(G)+3$. 
If $G$ contains two non-adjacent vertices $u$ and $v$ such that $N(v)\subseteq N(u)$, then $\chi(G)=\chi(G-v)$ and 
$\omega(G)=\omega(G-v)$, and we are done by applying the inductive hypothesis to $G-v$.
Therefore, we can assume that $G$ is a connected atom with no pair of comparable vertices.
If $G$ is perfect, then $\chi(G)=\omega(G)$ by \autoref{thm:SPGT}.
Otherwise the theorem follows from \autoref{thm:imperfect} below.
\end{proof}

\begin{theorem}\label{thm:imperfect}
Let $G$ be a connected atom with no pair of comparable vertices. If $G$ is ($P_6$, diamond)-free and imperfect, then $\chi(G)\le \omega(G)+3$.
\end{theorem}

The remainder of the section is devoted to the proof of \autoref{thm:imperfect}.
We begin with a simple lemma that will be useful later. 
A {\em matching} in a graph is a set of edges such that no two edges in the set meet a common vertex.

\begin{lemma}\label{lem:two cliques}
Let $G$ be a graph that can be partitioned into two cliques $X$ and $Y$ such that the edges between $X$ and $Y$ form a matching.
If $\max\{|X|,|Y|\}\le k$ for some integer $k\ge 2$, then $G$ is $k$-colourable.
\end{lemma}

\begin{proof}
Note that either $X$ or $Y$ is a maximum clique of $G$ unless $X$ and $Y$ are singletons in which case
the maximum size of a clique of $G$ is at most 2. Moreover, $G$ is perfect by \autoref{thm:SPGT}.
Since $\max\{|X|,|Y|\}\le k$ and $2\le k$, it follows that $G$ is $k$-colourable. 
\end{proof}

\begin{proof}[Proof of \autoref{thm:imperfect}]
Let $G = (V, E)$ be a graph satisfying the assumptions of the theorem.
Since $G$ is imperfect, it contains an induced $C_5$, say $Q=\{v_1,v_2,v_3,v_4,v_5\}$ (in order).
We partition $V(G)\setminus Q$ as in \autoref{sec:imperfect}.
Let $\omega:=\omega(G)$.
If $\omega\le 3$, then the theorem follows from a known result that every ($P_6$, diamond)-free graph without a $K_4$ is 6-colourable \cite{KM18}.
Therefore, we can assume that $\omega\ge 4$.
The idea is to colour $Q\cup A\cup B$, $C\cup F$, and $Z$ independently using
as few colours as possible. However, to obtain the optimal bound, we need to reuse colours in some smart way.
In particular, we show that we can reuse one colour from $Q\cup A\cup B$ on some $C_{i,i+2}$, so that
the remaining of $C\cup F$ can be coloured with only 3 colours (\autoref{lem:AB} and \autoref{lem:CFminusCi}). 
The proof of \autoref{lem:CFminusCi} relies on a computer search combined
with a well-known result characterizing 3-colourable $(P_6,K_3)$-free graphs (\autoref{thm:P6C3}).
See \autoref{fig:colouring} for a diagram illustrating our colouring of $G$ with $\omega+3$ colours.

\begin{figure}[h!]
\centering
\begin{tikzpicture}[scale=0.8]

\tikzstyle{vertex}=[draw, circle, fill=black!15]
\tikzstyle{set}=[draw,black!100,rounded corners]
%\draw[step=1cm,colour=gray] (-6,-6) grid (6,6);

\node[set] (C5AB) at (-3,0) {$Q\cup A\cup B$};
\node[set] (CF) at (-3,-5) {$(C\cup F)\setminus C_{i,i+2}$};
\node[set] (N) at (3,-5) {$Z$};
\node[set] (Ci) at (3,0) {$C_{i,i+2}$};

\draw (C5AB)--(CF)--(N);
\draw (C5AB)--(Ci)--(N);
\draw (Ci)--(CF);
\draw[dashed] (N)--(C5AB);

\node at (-3,-6) {\autoref{lem:CFminusCi}:  $\omega+1,\omega+2,\omega+3$};
\node at (3,-6) {\autoref{lem:colournonneighbours}:  $1,2,\ldots,\omega-1$};
\node at (3,1) {$\omega$};
\node at (-3,1) {$1,2,\ldots,\omega$};

\node at (-6,1) {\autoref{lem:AB}:};
\end{tikzpicture}
\caption{A $(\omega+3)$-colouring of $G$. A solid line means that the edges between the two sets are arbitrary and a dashed line means that
the two sets are anti-complete. }\label{fig:colouring}
\end{figure}
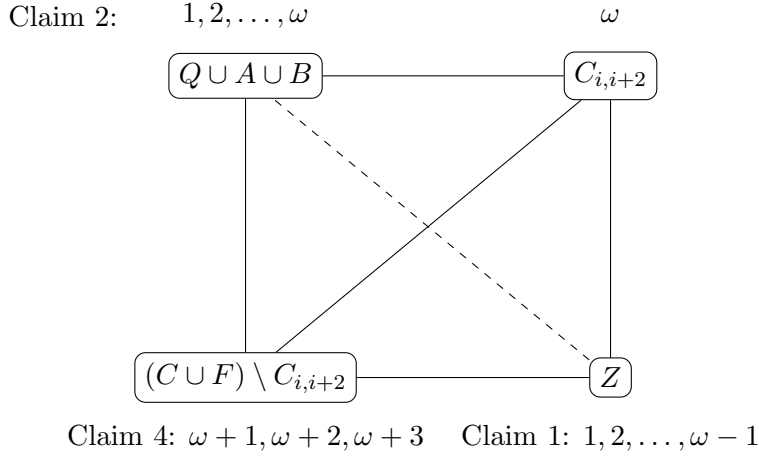

We first deal with the components of $Z$.

\begin{claim}\label{lem:colournonneighbours}
Each component of $Z$ is $(\omega-1)$-colourable.
\end{claim}

\begin{proof}
Let $K$ be an arbitrary component of $Z$. 
Suppose first that $K$ has a neighbour $c$ in $C_{i,i+2}$ for some $i$.
If $c$ is adjacent to $k_1\in K$ but not to $k_2\in K$ with $k_1k_2\in E$, then
$\{k_2,k_1,c,v_{i+2},v_{i+3},v_{i+4}\}$ induces a $P_6$. So, $c$ is complete to $K$.
Since $G$ is diamond-free, $K$ is a clique of size at most $\omega-1$. Thus, $K$ is $(\omega-1)$-colourable.

Suppose now that $K$ has no neighbour in $C$.  It follows from \ref{item:21} that  $K$ is anti-complete to $A\cup B$.
Since $G$ is connected, $K$ has a neighbour $f\in F$.
Let $L_1:=N(f)\cap K$ and $L_{i+1}:=N(L_i)\cap K$ for $i\ge 1$. 
Since $K$ is connected, $K=\bigcup_{i\ge 1}L_i$.
Suppose that $L_i$ contains a vertex $s_i$ for some $i\ge 3$. 
By definition, there is an induced path $f,s_1,\ldots,s_i$ such that $s_j\in L_j$ for each $1\le j\le i$.
Then for some $1\le k\le 5$, we have that $v_{k+1},v_{k}, f,s_1,\ldots,s_i$ contains an induced $P_6$, where $v_kf\in E$ and $v_{k+1}f\notin E$. 
This shows that $K=L_1\cup L_2$.
Since $G$ is diamond-free, each component of $L_1$ is $P_3$-free and thus is a clique.
Since $G$ is $P_6$-free, each vertex in $L_1$ is either complete or anti-complete to each component of $L_2$.
Moreover, if a component $X$ of $L_2$ has two neighbours in a component of $L_1$,
then two such neighbours, a vertex in $X$ and $f$ induce a diamond.

First, we claim that $L_2$ is stable. Suppose not. Let $X$ be a component of $L_2$ with at least two vertices $x$ and $x'$.
By definition, $X$ has a neighbour $y$ in some component $Y$ of $L_1$. If $X$ has a neighbour $y'\in L_1$ with $y'\neq y$,
then $y'$ is in a component other than $Y$, and so $\{y,y',x,x'\}$ induces a diamond. So, $y$ is the only neighbour of $X$ in $L_1$.
Since $G$ has no clique cutset, $X$ has a neighbour $f'\in F$. Note that there is an induced path $f,v_i,v_j,f'$ where $v_i,v_j\in Q$.
If $f'$ distinguishes an edge $xx'$ in $X$, then $f,v_i,v_j,f',x',x$ induces a $P_6$.
So, $f'$ is complete to $X$. If $f'y\notin E$, then $\{f',y,x,x'\}$ induces a diamond.
So, $f'y\in E$.  Note that the above argument works for each neighbour of $X$ in $F$.
Hence, if there is another neighbour $f''\in F$ of $X$, then $f'f''\notin E$ by \ref{item:13} and so $\{f',f'',x,x'\}$ induces a diamond.
This shows that $\{f',y\}$ is a clique cutset that separates $X$ from $G$, a contradiction.
This proves that $L_2$ is stable.

Secondly, we claim that if $L_2\neq \emptyset$, then each component of $L_1$ has size at most 2.
Let $x\in L_2$. If $x$ has no neighbour in $F$, then $N(x)\subseteq L_1\subseteq N(f)$.
This contradicts the assumption that $G$ has no pair of comparable vertices. So $x$ has a neighbour $f'\in F$.
Note that there is an induced path $f,v_i,v_j,f'$ where $v_i,v_j\in Q$.
For each non-neighbour $y\in L_1$ of $x$, we have that $\{y,f,v_i,v_j,f',x\}$ induces a $P_6$  unless $yf'\in E$.
This shows that $f'$ is adjacent to each non-neighbour of $x$ in $L_1$.
Since $x$ has at most one neighbour in each component of $L_1$, if a component of $L_1$ has size at least 3,
then $f'$ is adjacent to at least two vertices in that component, and so these two vertices, $f$ and $f'$ induce a diamond.
This proves the claim.

We now complete the proof using the above two claims.
If $L_2=\emptyset$, then $K=L_1$ is a clique of  size at most $\omega-1$,
and so is $(\omega-1)$-colourable. If $L_2\neq \emptyset$, then $K$ is 3-colourable by the two claims. 
Since $\omega\ge 4$, it follows that $K$ is $(\omega-1)$-colourable.
\end{proof}

Next we deal with $Q\cup A\cup B\cup C_{i,i+2}$ for some $i$.

\begin{claim}\label{lem:AB}
There exists an index $1\le i\le 5$ such that $Q\cup A\cup B\cup C_{i,i+2}$ can be coloured with $\omega$ colours
such that $C_{i,i+2}$ is monochromatic.
\end{claim}

\begin{proof}
 By \ref{item:5}, we may assume that $B=B_{2,3}\cup B_{4,5}$.
We consider several cases. In each case we give a desired colouring explicitly.
In the following, when we say that we colour a set with a certain colour, we mean that we colour each vertex in the set with that colour.
For convenience, we always colour $C_{i,i+2}$ with colour $\omega$ below.

\case{1} $A_1$ is not stable. 

By \ref{item:19}, we have that $A_3=A_4=B_{2,3}=B_{4,5}=\emptyset$.
Moreover, $A_1$ is anti-complete to $A_2\cup A_5$ by \ref{item:2}, and $A_2$ and $A_5$ are complete to each other by \ref{item:3}.
By \ref{item:19}, if $A_2$ is not stable, then $A_5$ is empty. This implies that
one of $A_2$ and $A_5$ is stable. 
By symmetry, we may assume that $A_5$ is stable.
We now colour $Q\cup A\cup B\cup C_{1,3}$ as follows.

$\bullet$ Colour $Q=v_1,v_2,v_3,v_4,v_5$ with colours 1, 2, 1, 2, 3 in order.

$\bullet$ Colour $A_5$ with colour 2.

$\bullet$ Colour each component of $A_1$ with colours in $\{2,3,\ldots,\omega\}$ using the smallest colour available.

$\bullet$ Colour each component of $A_2$ with colours in $\{1,3,\ldots,\omega\}$ using the smallest colour available.

$\bullet$ Colour $C_{1,3}$ with colour $\omega$.

We now show that this is a $\omega$-colouring of $Q\cup A\cup B\cup C_{1,3}$.
Observe first that each trivial component of $A_1$ is coloured with 2 and  each trivial component of $A_2$ is coloured with 1.
By \ref{item:1}, the colouring is proper on $Q\cup A\cup B$. It remains to show that
the colour of $C_{1,3}$ does not conflict with those colours of $A$. 
By \ref{item:11},  no vertex in $C_{1,3}$ can have a neighbour in a non-trivial component of $A_2$.
So, $C_{1,3}$ does not conflict with $A_2$. 
Suppose that there exists a vertex $a\in A_1$ with colour $\omega$ who has a neighbour  $c\in C_{1,3}$.
Let $K$ be the component of $A_1$ containing $a$. Then $c$ is complete to $K$ by \ref{item:9}.
This implies that $K\cup \{v_1,c\}$ is a clique and so $|K|\le \omega- 2$.
This contradicts that $a$ is coloured with colour $\omega$.
So,  $C_{1,3}$ does not conflict with $A_1$. This proves that the colouring is a proper colouring.

\medskip
\case{2} $A_1$ is stable but not empty.

By \ref{item:19}, we have that $A_3$ and $A_4$ are stable.
By \ref{item:20}, each of $B_{2,3}$ and $B_{4,5}$ contains at most one vertex.
Let $b_{2,3}$ and $b_{4,5}$ be the possible vertex in $B_{2,3}$ and $B_{4,5}$, respectively.
By \ref{item:19}, we have that one of $A_2$ and $A_5$ is stable. By symmetry, we may assume that
$A_5$ is stable. 
If $A_2$ is stable, then it is easy to verify that the following is a 3-colouring of $Q\cup A\cup B$:
$\{v_1,v_3\}\cup A_4\cup B_{4,5}$, $\{v_2,v_4\}\cup A_5\cup A_1$, $\{v_5\}\cup B_{2,3}\cup A_2\cup A_3$.
Since $\omega\ge 4$, one can extend this colouring to a $\omega$-colouring of $Q\cup A\cup B\cup C_{i,i+2}$
for any $i$.
We now assume that $A_2$ is not stable. By \ref{item:19}, we have that $A_4=A_5=\emptyset$.
We can colour $Q\cup A\cup B\cup C_{1,3}$ as follows.

$\bullet$ Colour $Q=v_1,v_2,v_3,v_4,v_5$ with colours 1, 2, 3, 1, 3 in order.

$\bullet$ Colour $A_1$ and $A_3$ with colours 3 and 1, respectively.

$\bullet$ Colour each component of $A_2$ with colours in $\{1,3,\ldots,\omega\}$ using the smallest colour available.

$\bullet$ Colour $b_{2,3}$ and $b_{4,5}$ with colours 1 and 2, respectively.

$\bullet$ Colour $C_{1,3}$ with colour $\omega$.

An argument similar to that in Case 1 shows that this is indeed an $\omega$-colouring of $Q\cup A\cup B\cup C_{1,3}$. 

\medskip
\case{3} $A_1$ is empty.
We further consider two subcases.

\medskip
\case{3.1} $A_2$ is not stable.

By \ref{item:19}, we have that $A_4=A_5=\emptyset$.
By \ref{item:20}, either $A_3$ is empty or $B_{4,5}$ has at most one vertex.

Suppose first that $A_3$ not empty. Then $B_{4,5}$ has at most one vertex. 
Let $b_{4,5}$ be the possible vertex in $B_{4,5}$.
Consider the following colouring of $Q\cup A\cup B\cup C_{1,3}$.

$\bullet$ Colour $Q=v_1,v_2,v_3,v_4,v_5$ with colours 1, 2, 3, 1, 3 in order.

$\bullet$ Colour each component of $A_2$ with colours in $\{1,3,\ldots,\omega\}$ 
using the smallest colour available.

$\bullet$ Colour each component of $A_3$ with colours in $\{1,2,4,\ldots,\omega\}$ 
using the smallest colour available.

$\bullet$ Colour $b_{4,5}$ with 2 if $b_{4,5}$ exists, and 
colour vertices in $B_{2,3}$ with colours in $\{1,4,\ldots,\omega\}$.

$\bullet$ Colour $C_{1,3}$ with colour $\omega$.

By \ref{item:4}, we have that $|B_{2,3}|\le \omega-2$.
By \ref{item:12}, we have that $C_{1,3}$ and $B_{2,3}$ are anti-complete.
By \ref{item:11}, we have that $C_{1,3}$ is anti-complete to each non-trivial component of $A_2$.
If $b_{4,5}$ exists, then $A_3$ is stable by \ref{item:20}. It then follows from the definition that
each vertex in $A_3$ is coloured with 1. One can easily verify that the above is a proper $\omega$-colouring of $Q\cup A\cup B\cup C_{1,3}$.
If $b_{4,5}$ does not exist, then an argument similar to that in Case 1 shows that
this is a proper $\omega$-colouring of $Q\cup A\cup B\cup C_{1,3}$.

Suppose now that $A_3$ is empty. 
Since $G$ is diamond-free, the edges between $B_{4,5}$ and $B_{2,3}$ form a matching.
For the same reason, the edges between $B_{4,5}$ and each component of $A_2$ form a matching.
Consider the following colouring of $Q\cup A\cup B\cup C_{1,3}$.

$\bullet$ Colour $Q=v_1,v_2,v_3,v_4,v_5$ with colours 3, $\omega$, 1, $\omega$, 1 in order.

$\bullet$ For each component $K$ of $A_2$, pick an arbitrary vertex $a_K$ in the component and colour it with 1.
By \autoref{lem:two cliques}, there exists a $(\omega-2)$-colouring of $B_{4,5}\cup (K\setminus a_K)$ using colours $2,3,\ldots,\omega-1$.

$\bullet$ By \autoref{lem:two cliques}, there exists a $(\omega-2)$-colouring of $B_{4,5}\cup B_{2,3}$ using colours $2,3,\ldots,\omega-1$.

$\bullet$ Colour $C_{1,3}$ with colour $\omega$.

Since $B_{2,3}$ and $A_2$ are anti-complete, the above colouring (by permuting colours in $A_2$) gives an $\omega$-colouring of $Q\cup A\cup B\cup C_{1,3}$.

\case{3.2} $A_2$ is stable. By symmetry, $A_5$ is stable.

Suppose first that $A_3$ is not stable. By \ref{item:19}, we have that $A_5=B_{4,5}=\emptyset$.
If $A_4$ is stable, one can easily verify that the following is an $\omega$-colouring of $Q\cup A\cup B\cup C_{2,4}$.

$\bullet$ Colour $Q=v_1,v_2,v_3,v_4,v_5$ with colours 1, 2, $\omega$, 3, 2 in order.

$\bullet$ Colour $A_2$ and $A_4$ with 1 and 2, respectively, and colour each component of $A_3$ with colours in $\{1,2,\ldots,\omega-1\}$.

$\bullet$ Colour vertices in $B_{2,3}$ with colours in $\{1,3,\ldots,\omega-1\}$.

$\bullet$ Colour $C_{2,4}$ with colour $\omega$.

If $A_4$ is not stable, then $B_{2,3}=\emptyset$ by \ref{item:19}. One can obtain a desired colouring as in Case 1.

\medskip
Now suppose that $A_3$ is stable. By symmetry, $A_4$ is stable. So, each $A_i$ is stable for $2\le i\le 5$.
We first claim that if both $A_2$ and $A_5$ are not empty, then each of $B_{2,3}$ and $B_{4,5}$ contains at most one vertex.
Let $a_2\in A_2$ and $a_5\in A_5$. By \ref{item:3}, it follows that $a_5a_2\in E$. If $b\in B_{2,3}$ is not adjacent to $a_5$,
then $\{b,v_3,v_4,v_5,a_5,a_2\}$ induces a $P_6$. So, $B_{2,3}$ is complete to $a_5$. 
Then $B_{2,3}$ contains at most one vertex, for otherwise two vertices in $B_{2,3}$, $a_5$ and $v_2$ induce a diamond.
Similarly, $B_{4,5}$ contains at most one vertex. 
This proves the claim.
Now if $A_2$ and $A_5$ are not empty, then since $\omega\ge 4$, the following is an $\omega$-colouring of $Q\cup A\cup B\cup C_{2,4}$:
$\{v_1,v_3,b_{4,5}\}\cup A_4\cup A_5$, $\{v_2,v_4\}$, $\{v_5,b_{2,3}\}\cup A_2\cup A_3$, and $C_{2,4}$.
So, we can assume by symmetry that $A_2=\emptyset$.
By \ref{item:20}, either $A_3=\emptyset$ or $B_{4,5}$ has at most one vertex.
One can easily verify that the following is an $\omega$-colouring of $Q\cup A\cup B\cup C_{2,4}$.

$\bullet$ Colour $Q=v_1,v_2,v_3,v_4,v_5$ with colours 3, 1, 2, 1, 2 in order.

$\bullet$ Colour $A_4$ and $A_5$ with 2 and 1, respectively, and colour $A_3$ with 3 if $A_3\neq \emptyset$.

$\bullet$ By \autoref{lem:two cliques}, there exists a $(\omega-2)$-colouring of $B_{4,5}\cup B_{2,3}$
using colours in $\{3,\ldots,\omega\}$. If $B_{4,5}$ contains  at most one vertex $b_{4,5}$, 
we may assume that $b_{4,5}$ is coloured with colour $\omega$.

$\bullet$ Colour $C_{2,4}$ with colour $\omega$.
\end{proof}

Finally, we deal with $C\cup F$.

\begin{claim}\label{lem:CF}
The subgraph $C\cup F$ is triangle-free.
\end{claim}

\begin{proof}
Suppose, by contradiction, that $C \cup F$ contains a triangle $T$ with vertices $h_i$ for $i=1,2,3$. 
Since $F$ is stable by \ref{item:13}, it follows that $T$ contains at least two vertices from $C$. 
Moreover, vertices in $T\cap C$ are in different $C_{i,i+2}$,  since each $C_{i,i+2}$ is stable by \ref{item:8}.
If $T$ contains a vertex of $F$, then the other two vertices of $T$ are from $C_{i-1,i+1}$ by \ref{item:18}.
But this contradicts the fact that $C_{i-1,i+1}$ is stable. So, all vertices of $T$ are in $C$.
If the three vertices of $T$ are from $C_{i,i+2}$, $C_{i+1,i+3}$, and $C_{i+2,i+4}$ for some $i$, 
then $\{h_1, h_2, h_3, v_{i+2}\}$ induces a diamond in $G$.
If the three vertices of $T$ are from $C_{i,i+2}$, $C_{i,i-2}$, and $C_{i+2,i+4}$ for some $i$, 
then $\{h_1, h_2, h_3, v_{i+2}\}$ induces a diamond in $G$. 
\end{proof}

\begin{claim}\label{lem:CFminusCi}
For each $1\le i\le 5$, the subgraph $(C\cup F)\setminus C_{i,i+2}$ is $3$-colourable.
\end{claim}

\begin{proof}
We show via a computer program that $(C\cup F)\setminus C_{i,i+2}$ does not contain the Gr\"{o}tzsch graph as an induced subgraph.
Since $C\cup F$ is triangle-free by \autoref{lem:CF}, it follows from \autoref{thm:P6C3} that  $(C\cup F)\setminus C_{i,i+2}$ is 3-colourable.

We now explain the algorithm and give the pesudocode.
Let $H$ be an induced copy of the Gr\"{o}tzsch graph (see \autoref{fig:Grotzsch}).
For each vertex $v\in V(H)$, a {\em label} of $v$ is an element in the set
$S=\{1,2,3,4,5,13,14,24,25,35\}$. The meaning of the label of $v$ is to indicate
where $v$ comes from. For example, if the label of $v$ is 1, it indicates that $v\in F_1$,
and if the label of $v$ is 13, it indicates that $v\in C_{1,3}$.
A {\em labelling} of $H$ is a function $\mathcal{L}:V(H)\rightarrow S$.
We denote by $H_{\mathcal{L}}$ the copy of $H$ with labelling $\mathcal{L}$.
For a labelling  $\mathcal{L}$ of $H$, we say that $\mathcal{L}$ is {\em valid}
if the graph obtained by taking the union of $H$ and $Q$, 
where the edges between $H$ and $Q$ are connected according to $\mathcal{L}$,
is ($P_6$, diamond)-free.
%To see which labellings of $H$ are valid , one could try
%all possible combinations of labels of vertices. The number of such combinations is $11^{10}$,
%since there are 10 choices for each vertex $v\in V(H)$. Therefore, a brute-force algorithm does not work.

We use a simple recursive algorithm that uses certain reduction rules to find all valid labellings of $H$.
The algorithm \textsc{Main} (see \autoref{alg:main}) takes two parameters $\mathcal{L}$ and $\mathcal{F}$
as inputs, where $\mathcal{L}$ is a function from $V(H)$ to the power set $2^S$ of $S$ and $\mathcal{F}$ is a set to
store valid labellings of $H$, and returns a set $\mathcal{F}$ of valid labellings of $H$ where the label of each $v\in V(H)$ is in $\mathcal{L}(v)$.
The algorithm recursively checks if a vertex $v\in V(H)$ can be labelled with label $\ell$  for each label $\ell\in \mathcal{L}(v)$. 
Once a label $\ell$ is assigned to $v$, the algorithm calls the subroutine \textsc{UpdateLabels} (see \autoref{alg:subroutine}) to
update possible labels for other vertices using certain reduction rules (see \ref{item:R1}-\ref{item:R3} below). 
If at some point $\mathcal{L}(v)$ becomes empty for some $v\in V(H)$, we discard the search, since
the current labelling is not valid. If at some point $\mathcal{L}(v)$ becomes a singleton for each $v\in V(H)$, then $\mathcal{L}$ is a 
labelling of $H$. The algorthm then checks whether it is valid. If so,  the labelling is added to the list $\mathcal{F}$ of valid labelling,
and is discarded otherwise.

\begin{enumerate}[label= Rule \arabic*]

\item If $v\in V(H)$ has label $i(i+2)$, then the label of each neighbour of $v$ is in $(\{13,14,24,25,35\}\setminus \{i(i+2)\})\cup \{i+1\}$.
\label{item:R1}

This follows from \ref{item:8} and \ref{item:18}. \e

\item If $v\in V(H)$ has label $i$, then the only possible label for each neighbour of $v$ is $(i-1)(i+1)$.
\label{item:R2}

This follows from \ref{item:8}, \ref{item:13} and \ref{item:18}. \e

\item If $v\in V(H)$ has two neighbours $u_1$ and $u_2$ whose labels are $i(i+2)$, 
then the label of $v$ cannot contain the numbers $i$ and $i+2$ in its label.
\label{item:R3}

Suppose not. Then $v$ is adjacent to $v_i$ or $v_{i+2}$. But now $\{v,u_1,u_2,v_i\}$ or $\{v,u_1,u_2,v_{i+2}\}$ induces a diamond. \e

\end{enumerate}

%%%%%%%%%%%%%%%%%%%%%%%%% Main
\begin{algorithm}[H]

\SetAlgoLined

\KwIn{A function $\mathcal{L} \colon V(H)\rightarrow 2^S$.}

\KwOut{All valid labellings of $H$ such that the label of each $v\in V(H)$ is in $\mathcal{L}(v)$.}

\medskip
\tcp*[h]{Base cases}

\If{{\em there exists a vertex} $v\in V(H)$ {\em such that }$\mathcal{L}(v)=\emptyset$}
{
	return $\mathcal{F}$;
}

\ElseIf{$|\mathcal{L}(v)|=1$ {\em for each $v\in V(H)$}}
{
	Let $H'=H\cup Q$ where the edges between $H$ and $Q$ are constructed according to $\mathcal{L}$.
	
	\If{{\em $H'$ is not ($P_6$,diamond)-free}}
	{
		return $\mathcal{F}$;
	}
	\Else
	{
		return $\mathcal{F}\cup \{H_{\mathcal{L}}\}$;
	}
}

\medskip
\tcp*[h]{Recursive call}

\medskip
\Else
{
	\For{{\em each $v\in V(H)$ with $|\mathcal{L}(v)|\ge 2$}}
	{
		\For{{\em each label $\ell\in \mathcal{L}(v)$}}
		{
			$\mathcal{L'}:=\mathcal{L}$;
			
			$\mathcal{L'}(v):=\ell$;
			
			\textsc{UpdateLabels}($\mathcal{L'}$);
			
			return \textsc{Main}($\mathcal{L'}$, $\mathcal{F}$);
		}
	}
}

\medskip
\caption{A recursive algorithm \textsc{Main}($\mathcal{L}$, $\mathcal{F}$).} \label{alg:main}
\end{algorithm}

%%%%%%%%%%%%%%%%%%%%%%%%% Update Labels
\begin{algorithm}[H]

\SetAlgoLined

%\KwIn{A function $\mathcal{L} \colon V(H)\rightarrow 2^S$.}
%
%\KwOut{All valid labelling of $H$.}

\For{{\em each $v\in V(H)$}}
{
	\tcp*[h]{\ref{item:R1}}
	
	\If{$\mathcal{L}(v)=i(i+2)$}
	{
		\For{{\em each $u\in N_H(v)$}}
		{
			$\mathcal{L}(u):=\mathcal{L}(u)\cap ((\{13,14,24,25,35\}\setminus \{i(i+2)\})\cup \{i+1\})$;
		}
	}
	
	\tcp*[h]{\ref{item:R2}}
	
	\If{$\mathcal{L}(v)=i$}
	{
		\For{{\em each $u\in N_H(v)$}}
		{
			$\mathcal{L}(u):=\mathcal{L}(u)\cap \{(i-1)(i+1)\}$;
		}
	}
	
	\tcp*[h]{\ref{item:R3}}
	
	\If{{\em there exist $u_1,u_2\in N_H(v)$ such that $\mathcal{L}(u_1)=\mathcal{L}(u_2)=i(i+2)$}}
	{
		$\mathcal{L}(v):=\mathcal{L}(v)\cap (S\setminus \{i(i+2), (i-2)(i),(i+2)(i+4),i,i+2\})$;
	}
}

\medskip
\caption{The subroutine \textsc{UpdateLabels}($\mathcal{L}$).} \label{alg:subroutine}
\end{algorithm}

By symmetry, we can assume the label of vertex 0 in $H$ is $\{1\}$ or $\{25\}$. To find all valid labellings of $H$, we make a single call to \textsc{Main}($\mathcal{L}^*$, $\emptyset$)
where $\mathcal{L}^*:V(H)\rightarrow 2^S$ such that $\mathcal{L}^*(0) = \{1, 25\}$ and $\mathcal{L}^*(v)= S$ for each $v \in V(H)\setminus \{0\}$.
The algorithm returns 20 valid labellings of $H$. Each of these labellings needs a vertex from $C_{i,i+2}$
for each $1\le i\le 5$. This shows that $(C\cup F)\setminus C_{i,i+2}$ does not contain the Gr\"{o}tzsch graph.
The output of the algorithm is given in the Appendix. 
%and the source file of the program can be downloaded at \href{www.google.ca}{source file}.
\end{proof}

We now give a $(\omega+3)$-colouring of $G$.
By \autoref{lem:AB}, there exists an index $i$ such that $Q\cup A\cup B\cup C_{i,i+2}$ can be coloured with
colours from $\{1,2,\ldots,\omega\}$, and that $C_{i,i+2}$ is coloured with colour $\omega$.
By \autoref{lem:colournonneighbours}, we can colour each component of $Z$ with colours from $\{1,2,\ldots,\omega-1\}$.
By \autoref{lem:CFminusCi}, we can colour $(C\cup F)\setminus C_{i,i+2}$ with colours $\omega+1,\omega+2,\omega+3$.
This is a $(\omega+3)$-colouring of $G$ (See \autoref{fig:colouring}).
\end{proof}

\section{Conclusion}\label{sec:conclude}

In this paper, we proved that each ($P_6$, diamond)-free graph $G$ satisfies $\chi(G)\le \omega(G)+3$.
This answers an open question in \cite{KM18} and gives an optimal $\chi$-bound for the class.
It is not difficult to see that one can turn our proof into a polynomial-time algorithm for colouring a 
($P_6$, diamond)-free graph $G$ using $\omega(G)+3$ colours. A natural question is whether
one can decide the chromatic number of these graphs in polynomial time. To answer this question, it may
be useful to consider whether there exists a structure theorem for the class of 
($P_6$, diamond)-free graphs. We point out that Chudnovsky, Seymour, Spirkl and Zhong \cite{CSSZ18} give such a structure theorem
for a subclass, namely $(P_6,K_3)$-free graphs.

\section*{Appendix}

We use an array of size 11 to represent a valid labelling of $H$.
The array is indexed by the vertices of $H$ (see \autoref{fig:Grotzsch}), which means that
the first element of the array is the label for vertex 0 of $H$,
and the second element of the array is the label for vertex 1 of $H$, etc.
The algorithm returns the following 20 valid labellings of $H$.

\noindent 
[[25], [13], [13], [14], [13], [14], [25], [35], [24], [35], [24]]\break
[[25], [13], [13], [14], [13], [14], [35], [25], [24], [35], [24]]\break
[[25], [13], [14], [13], [13], [14], [35], [24], [25], [35], [24]]\break
[[25], [13], [14], [13], [13], [14], [35], [24], [35], [25], [24]]\break
[[25], [13], [14], [13], [14], [13], [25], [24], [35], [24], [35]]\break
[[25], [13], [14], [13], [14], [13], [35], [24], [35], [24], [25]]\break
[[25], [13], [14], [13], [14], [14], [35], [24], [35], [24], [25]]\break
[[25], [13], [14], [13], [14], [14], [35], [24], [35], [25], [24]]\break
[[25], [13], [14], [14], [13], [14], [35], [24], [25], [35], [24]]\break
[[25], [13], [14], [14], [13], [14], [35], [25], [24], [35], [24]]\break
[[25], [14], [13], [13], [14], [13], [24], [25], [35], [24], [35]]\break
[[25], [14], [13], [13], [14], [13], [24], [35], [25], [24], [35]]\break
[[25], [14], [13], [14], [13], [13], [24], [35], [24], [25], [35]]\break
[[25], [14], [13], [14], [13], [13], [24], [35], [24], [35], [25]]\break
[[25], [14], [13], [14], [13], [14], [24], [35], [24], [35], [25]]\break
[[25], [14], [13], [14], [13], [14], [25], [35], [24], [35], [24]]\break
[[25], [14], [13], [14], [14], [13], [24], [35], [24], [25], [35]]\break
[[25], [14], [13], [14], [14], [13], [24], [35], [25], [24], [35]]\break
[[25], [14], [14], [13], [14], [13], [24], [25], [35], [24], [35]]\break
[[25], [14], [14], [13], [14], [13], [25], [24], [35], [24], [35]]\break

\end{document}